\documentclass[11pt]{article}
\usepackage{epic,eepic}
\usepackage{amsmath,amssymb,amsfonts}
\begin{document}
\newtheorem{theorem}{Theorem}[section]
\newtheorem{lemma}[theorem]{Lemma}
\newtheorem{proposition}[theorem]{Proposition}
\newtheorem{definition}[theorem]{Definition}
\def\emptyset{\varnothing}
\def\setminus{\smallsetminus}
\def\loc{{\mathrm{loc}}}
\def\Irr{{\mathrm{Irr}}}
\def\Rep{{\mathrm{Rep}}}
\def\End{{\mathrm{End}}}
\def\Vec{{\mathrm{Vec}}}
\def\opp{{\mathrm{opp}}}
\def\id{{\mathrm{id}}}
\def\A{{\mathcal{A}}}
\def\B{{\mathcal{B}}}
\def\C{{\mathcal{C}}}
\def\D{{\mathcal{D}}}
\def\N{{\mathbb{N}}}
\def\Q{{\mathbb{Q}}}
\def\R{{\mathbb{R}}}
\def\Z{{\mathbb{Z}}}
\def\a{{\alpha}}
\def\e{{\varepsilon}}
\def\la{{\lambda}}
\def\th{{\theta}}
\def\isom{{\cong}}
\newcommand{\Hom}{\mathop{\mathrm{Hom}}\nolimits}
\def\qed{{\unskip\nobreak\hfil\penalty50
\hskip2em\hbox{}\nobreak\hfil$\square$
\parfillskip=0pt \finalhyphendemerits=0\par}\medskip}
\def\proof{\trivlist \item[\hskip \labelsep{\bf Proof.\ }]}
\def\endproof{\null\hfill\qed\endtrivlist\noindent}

\title{A relative tensor product of subfactors\\
over a modular tensor category}
\author{
{\sc Yasuyuki Kawahigashi}\footnote{Supported in part by 
Research Grants and the Grants-in-Aid
for Scientific Research, JSPS.}\\
{\small Graduate School of Mathematical Sciences}\\
{\small The University of Tokyo, Komaba, Tokyo, 153-8914, Japan}
\\[0,05cm]
{\small and}
\\[0,05cm]
{\small Kavli IPMU (WPI), the University of Tokyo}\\
{\small 5-1-5 Kashiwanoha, Kashiwa, 277-8583, Japan}\\
{\small e-mail: {\tt yasuyuki@ms.u-tokyo.ac.jp}}}
\maketitle{}
\begin{abstract}
We define and study a certain relative tensor product of subfactors
over a modular tensor category.  This gives a relative tensor
product of two completely rational heterotic full local conformal nets
with trivial superselection structures over a common chiral
representation category.  
In particular, we have a new realization of fusion rules of modular
invariants.  This also gives a mathematical
definition of a composition of two gapped domain walls between
topological phases.  
\end{abstract}

\section{Introduction}

The theory of subfactors due to Jones \cite{J} has been a very powerful
tool in conformal field theory.  We study some aspects of full conformal
field theory from a viewpoint of subfactors and modular tensor
categories.  (We consider only unitary modular tensor categories
in this paper.)

We are interested in a subfactor $N\subset M$ with finite
Jones index $[M:N]$.  In conformal field theory, it is often useful
to formulate a subfactor $N\subset M$
in terms of a $Q$-system $\Theta=(\theta,w,x)$
where $\theta$ is an
endomorphism of a type III factor $N$ with separable predual
and $w\in\Hom(\id,\theta)$, $x\in\Hom(\theta,\theta^2)$
as in \cite{L}.  When $\theta$ is an object of an abstract
modular tensor category $\C$, we say
$\Theta$ is a $Q$-system on $\C$.  (Note that any modular
tensor category is realized as a subcategory of $\End(N)$ for
a type III factor $N$.)  It is also often called
a $C^*$-Frobenius algebra on $\C$.  When we have 
$x=\e(\th,\th)x$, where $\e$ denotes the braiding,
we say that the $Q$-system $\Theta$ is
local.  It is also often said that it is commutative.
We say $\Theta$ is Lagrangian if we have
$(\dim\th)^2=\dim\C$.  (See \cite[page 153]{DMNO} for the
origin of this terminology.) See \cite{K1} and references therein
for more on subfactors and tensor categories.  Our
basic reference on modular categories is \cite{BK}.
See \cite{EK} for basics of subfactor theory.

Let $\{A(I)\}$ be a completely rational local conformal net
in the sense of \cite{KLM}, \cite{KL1}, and let $\C$ be
the Doplicher-Haag-Roberts representation category of
$\{A(I)\}$.  (It is a modular tensor category by \cite{KLM}.)
A maximal full conformal field theory in the sense of \cite{KL2}
is given by a local Lagrangian $Q$-system on $\C\boxtimes\C^\opp$
as in \cite{KL2}, where ``opp'' means the opposite modular
tensor category for which the braiding is reversed.
(Also see \cite[Proposition 6.7]{BKL}.)  Let 
$\th=\bigoplus_{\la\in{\Irr}(\C),
\mu\in{\Irr}(\C^\opp)}
Z_{\la\mu}\la\boxtimes\bar\mu$ be the object of such a
$Q$-system on $\C\boxtimes\C^\opp$, where ``Irr'' means the
set of equivalence classes of simple objects in the
modular tensor category.  The matrix $Z=(Z_{\la\mu})$ is
then a modular invariant in the sense that it commutes with
the $S$- and $T$-matrices arising from $\C$ as in 
\cite[Proposition 6.6]{BKL}.  Suppose we have two such
modular invariants $(Z^1_{\la\mu})$ and $(Z^2_{\mu\nu})$.
Then the matrix product $Z^1Z^2$ clearly satisfies the
properties of the modular invariant except for the 
normalization condition $Z_{00}=1$ where $0$ denotes the
identity object of the modular tensor category $\C$.
It is sometimes possible to have a decomposition
$Z^1 Z^2=\sum_i Z^{3,i}$ into modular invariants $Z^{3,i}$.
Such decomposition rules of matrix products have been studied
under the name of fusion rules of modular invariants
in \cite{E}, \cite[Section 3.1]{EP},
\cite[Remark 5.4 (iii)]{FRS}.
We have a machinery of $\a$-induction for subfactors
as in \cite{LR}, 
\cite{BE1}, 
\cite{BEK1}, \cite{BEK2},\cite{BEK3}, and it
produces a modular invariant as in \cite{BEK1}.
It gives a $Q$-system on $\C\boxtimes\C^\opp$
as in \cite{R1}, and
this is a general form of a maximal full conformal
field theory on $\C\boxtimes\C^\opp$
as in \cite[Proposition 6.7]{BKL}.
The results in \cite[Section 3.1]{EP},
\cite[Remark 5.4 (iii)]{FRS} say that a braided product
of $Q$-systems on $\C$ gives a fusion rule of
the corresponding $Q$-systems on $\C\boxtimes\C^\opp$.
In this way, we indirectly have an irreducible
decomposition of a certain relative tensor product
of two local irreducible Lagrangian $Q$-systems
on $\C\boxtimes\C^\opp$.

One typical example of such fusion rules is given as
follows.  Let $\C$ be the modular tensor category
corresponding to the WZW-model $SU(2)_{17}$.
Then by \cite[Page 202]{O} (and also by
\cite[Theorem 2.1]{KLPR} and 
\cite[Proposition 6.7]{BKL}), we have exactly three
irreducible local Lagrangian $Q$-systems on
$\C\boxtimes \C^\opp$ and they are labeled with
$A_{17},D_{10},E_7$ as in \cite{CIZ}.  (These
labels are for the modular invariant matrices.
The label $A_{17}$ corresponds to the identity 
matrix.)  Their nontrivial fusion rules are 
as follows by \cite[Section 5.1]{E},
\cite[Remark 5.4 (iii)]{FRS}.

\begin{align*}
D_{10}\otimes D_{10}&=2D_{10},\\
D_{10}\otimes E_7 &= E_7\otimes D_{10}=2E_7,\\
E_7\otimes E_7 &=D_{10}\oplus E_7.
\end{align*}

We would like to extend this relative product to the
irreducible local Lagrangian $Q$-systems on
$\C_1\boxtimes\C_2^\opp$ and 
$\C_2\boxtimes\C_3^\opp$ in this paper where 
$\C_1,\C_2,\C_3$ can be different.  This setting 
corresponds to a heterotic full conformal field theory.

The author thanks M. Bischoff, L. Kong, R. Longo,
K.-H. Rehren and Z. Wang
for useful discussions and comments.
Parts of this work were done at 
Instituto Superior T\'ecnico, Universidade de Lisboa
and Microsoft Research Station Q at Santa Barbara.
The author thanks the both institutions for
their hospitality.

\section{A relative tensor product of $Q$-systems}

We consider a $Q$-system $\Theta=(\theta,w,x)$ where $\theta$ is an
endomorphism of a type III factor $N$ with separable predual
and $w\in\Hom(\id,\theta)$, $x\in\Hom(\theta,\theta^2)$.
We adapt \cite[Definition 3.8]{BKLR}, which means that such
a $Q$-system corresponds to an inclusion $N\subset M$ where
$M$ may not be a factor.  We have $N'\cap M={\mathbb C}$ if
and only if the $Q$-system $\Theta$ is irreducible.

We recall the following proposition in \cite{M}.
(Also see \cite[Proposition 3.7, Corollary 3.8]{DNO}.)

\begin{proposition}
\label{modular}
Let $\Theta=(\theta, w, x)$ be an irreducible local
$Q$-system where
$\theta$ is of the form $\bigoplus_{\la\in{\Irr}(\C_1),
\mu\in{\Irr}(\C_2^\opp)}
Z_{\la\mu}\la\boxtimes\bar\mu$ for some modular
tensor categories $\C_1,\C_2$.  Then it is Lagrangian 
if and only if we have the modular invariance property
$S_{\C_1} Z=Z S_{\C_2}$
and $T_{\C_1} Z=Z T_{\C_2}$ for the matrix $Z=(Z_{\la\mu})$, 
where $S_{\C_1}, S_{\C_2}, T_{\C_1}, T_{\C_2}$ are the 
$S$-matrix for $\C_1$, $S$-matrix for $\C_2$, 
$T$-matrix for $\C_1$ and
$T$-matrix for $\C_2$, respectively.
\end{proposition}

This was first raised as a problem in
\cite[Section 3]{R2} in the context of full conformal field theory,
and proved by M\"uger \cite{M} and an unpublished manuscript of Longo
and the author.  This is valid in a general context of a modular
tensor category.  (Also see \cite[Proposition 5.2]{BKL}.)

Let $(\theta, w, x)$ be a $Q$-system where
$\theta$ is of the form $\bigoplus_{\la\in{\Irr}(\C_1),
\mu\in{\Irr}(\C_1),\nu\in{\Irr}(\C_2)}
Z_{\la\mu\nu}\la\boxtimes\mu\boxtimes\nu$ for some modular
tensor categories $\C_1,\C_2$.
By applying the functor $T$ to the $\C_1$ component
as in \cite[Section 4.1]{KR},
\cite[Section 4.2]{BKL}, we obtain a new $Q$-system
$T(\Theta)=(T(\th),w_T,x_T)$ where $T(\th)=
\bigoplus_{\la\in{\Irr}(\C_1),\mu\in{\Irr}(\C_1),
\nu\in{\Irr}(\C_2)}
Z_{\la\mu\nu}\la\mu\boxtimes\nu$.  (We need the
braiding structure of $\C_1$ in order to define
$x_T$.)  Note that
even if $\Theta$ is irreducible, $T(\Theta)$ is
not irreducible in general.

Let $(\theta, w, x)$ be another $Q$-system where
$\theta$ is of the form $\bigoplus_{\la\in{\Irr}(\C_1),
\mu\in{\Irr}(\C_2)}
Z_{\la\mu}\la\boxtimes\mu$ for some modular
tensor categories $\C_1,\C_2$.  By applying
\cite[Corollary 3.10]{ILP}, we have a new
$Q$-system $(\th_1,w_1,s_1)$ with
$\th=\bigoplus_{\la\in{\Irr}(\C_1)}Z_{\la0}\la$
where $0$ denotes the identity object of $\C_2$.
We call it the restriction of $\Theta$ to $\C_1$.

Now let  $\Theta_1=(\theta_1, w_1, x_1)$ 
and $\Theta_2=(\theta_2, w_2, x_2)$ be $Q$-systems where
$$\theta_1=\bigoplus_{\la\in{\Irr}(\C_1),\mu\in{\Irr}(\C_2^\opp)}
Z^1_{\la\mu}\la\boxtimes\bar\mu$$ on $\C_1\boxtimes\C_2^\opp$ and
$$\theta_2=\bigoplus_{\mu\in{\Irr}(\C_2),\nu\in{\Irr}(\C_3^\opp)}
Z^2_{\mu\nu}\mu\boxtimes\bar\nu$$ on $\C_2\boxtimes\C_3^\opp$
for some modular tensor categories $\C_1,\C_2,\C_3$.

Let $\Theta_1\boxtimes\Theta_2$ be the tensor product of
the two $Q$-systems for which the object is given by
$$\bigoplus_{\la\in{\Irr}(\C_1),\mu\in{\Irr}(\C_2^\opp),
\mu'\in{\Irr}(\C_2),\nu\in{\Irr}(\C_3^\opp)}
Z^1_{\la\mu}Z^2_{\mu'\nu}
\la\boxtimes\bar\mu\boxtimes\mu'\boxtimes\bar\nu.$$
By applying the $T$ functor to the $\C_2$ components,
we obtain a new $Q$-system whose object is
$$\bigoplus_{\la\in{\Irr}(\C_1),\mu\in{\Irr}(\C_2^\opp),
\mu'\in{\Irr}(\C_2),\nu\in{\Irr}(\C_3^\opp)}
Z^1_{\la\mu}Z^2_{\mu'\nu}
\la\boxtimes\bar\mu\mu'\boxtimes\bar\nu.$$
By restricting this $Q$-system to $\C_1\boxtimes\C_3^\opp$, we
obtain a new $Q$-system whose object is 
$$\bigoplus_{\la\in{\Irr}(\C_1),\mu\in{\Irr}(\C_2),
\nu\in{\Irr}(\C_3^\opp)}
Z^1_{\la\mu}Z^2_{\mu\nu}\la\boxtimes\bar\nu.$$

\begin{definition}
We call the above $Q$-system the relative tensor product of
$\Theta_1$ and $\Theta_2$ over $\C_2$ and write
$\Theta_1\otimes_{\C_2}\Theta_2$.
\end{definition}

From the definition, it is easy to see the following.

\begin{proposition}
The relative tensor product operation is associative.
\end{proposition}

To apply this notion to a full conformal field theory,
we need the following.

\begin{proposition}
If two $Q$-systems are both local, then the relative tensor
product $\Theta_1\otimes_{\C_2}\Theta_2$ is also local.
\end{proposition}

\begin{proof}
For notational simplicity, we may treat $\C_1\boxtimes\C_3^\opp$
as a single modular tensor category, so we simply write $\C_1$
for $\C_1\boxtimes\C_3^\opp$ as if
$\C_3$ were the trivial modular tensor category
$\Vec$ of finite dimensional Hilbert spaces.

Locality of the tensor product $Q$-system $\Theta_1\boxtimes\Theta_2$
is represented as in Fig. \ref{local1}.  
(We follow the graphical convention of \cite[Section 3]{BEK1}, but  
compose morphisms from the bottom to the top, which is a
converse direction to the one in \cite[Section 3]{BEK1}.)
In this picture, the triple points
on the left hand side denote $x_1,x_2,x_2$, respectively.  The second
braiding on the right hand side is reversed because we have $\C_2^\opp$
for this component.

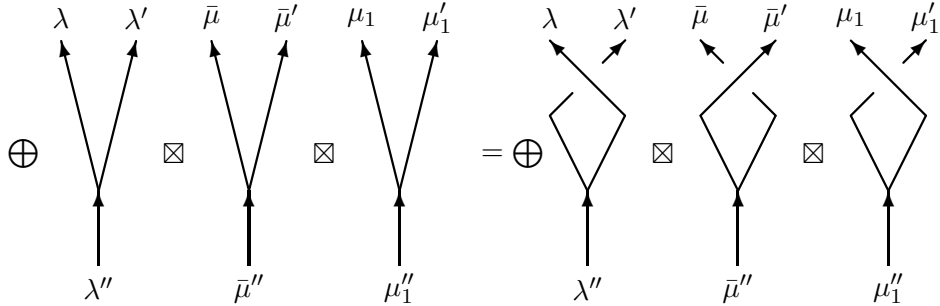
\begin{figure}[htb]
\begin{center}
\unitlength 1mm
\begin{picture}(140,50)
\thicklines
\multiput(20,10)(20,0){3}{\vector(0,1){10}}
\multiput(85,10)(20,0){3}{\vector(0,1){10}}
\multiput(20,20)(20,0){3}{\vector(-1,4){5}}
\multiput(20,20)(20,0){3}{\vector(1,4){5}}
\multiput(85,20)(20,0){3}{\line(-1,2){5}}
\multiput(85,20)(20,0){3}{\line(1,2){5}}
\multiput(90,30)(40,0){2}{\vector(-1,1){10}}
\put(100,30){\vector(1,1){10}}
\multiput(80,30)(40,0){2}{\line(1,1){3}}
\multiput(87,37)(40,0){2}{\vector(1,1){3}}
\put(110,30){\line(-1,1){3}}
\put(103,37){\vector(-1,1){3}}
\put(10,25){\makebox(0,0){$\bigoplus$}}
\put(75,25){\makebox(0,0){$=\bigoplus$}}
\put(30,25){\makebox(0,0){$\boxtimes$}}
\put(50,25){\makebox(0,0){$\boxtimes$}}
\put(95,25){\makebox(0,0){$\boxtimes$}}
\put(115,25){\makebox(0,0){$\boxtimes$}}
\put(20,7){\makebox(0,0){$\lambda''$}}
\put(15,43){\makebox(0,0){$\lambda$}}
\put(25,43){\makebox(0,0){$\lambda'$}}
\put(40,7){\makebox(0,0){$\bar\mu''$}}
\put(35,43){\makebox(0,0){$\bar\mu$}}
\put(45,43){\makebox(0,0){$\bar\mu'$}}
\put(60,7){\makebox(0,0){$\mu''_1$}}
\put(55,43){\makebox(0,0){$\mu_1$}}
\put(65,43){\makebox(0,0){$\mu'_1$}}
\put(85,7){\makebox(0,0){$\lambda''$}}
\put(80,43){\makebox(0,0){$\lambda$}}
\put(90,43){\makebox(0,0){$\lambda'$}}
\put(105,7){\makebox(0,0){$\bar\mu''$}}
\put(100,43){\makebox(0,0){$\bar\mu$}}
\put(110,43){\makebox(0,0){$\bar\mu'$}}
\put(125,7){\makebox(0,0){$\mu''_1$}}
\put(120,43){\makebox(0,0){$\mu_1$}}
\put(130,43){\makebox(0,0){$\mu'_1$}}
\end{picture}
\end{center}
\caption{Locality (1)}
\label{local1}
\end{figure}

From Fig. \ref{local1}, we connect the wires $\bar\mu''$ and
$\mu''_1$, the wires $\bar\mu$ and $\mu_1$, and the
wires $\bar\mu'$ and $\mu'_1$ on the both hand sides
so that the wires connecting
$\bar\mu$ and $\mu_1$ go over the ones connecting
$\bar\mu'$ and $\mu'_1$.  Then we obtain Fig. \ref{local2}.
Then the Reidemeister move II on the most right picture of
Fig. \ref{local2} produces Fig. \ref{local3}.

\begin{figure}[htb]
\begin{center}
\unitlength 1mm
\begin{picture}(135,50)
\thicklines
\put(20,10){\vector(0,1){10}}
\put(75,10){\vector(0,1){10}}
\put(20,20){\vector(-1,4){5}}
\put(20,20){\vector(1,4){5}}
\put(75,20){\line(-1,2){5}}
\put(75,20){\line(1,2){5}}
\put(80,30){\vector(-1,1){10}}
\put(70,30){\line(1,1){3}}
\put(77,37){\vector(1,1){3}}
\put(10,25){\makebox(0,0){$\bigoplus$}}
\put(65,25){\makebox(0,0){$=\bigoplus$}}
\put(30,25){\makebox(0,0){$\boxtimes$}}
\put(85,25){\makebox(0,0){$\boxtimes$}}
\put(35,30){\vector(1,-1){10}}
\put(50,25){\vector(-1,1){10}}
\put(55,30){\vector(-1,1){5}}
\put(35,30){\line(1,1){5}}
\put(40,25){\line(1,1){3}}
\put(47,32){\line(1,1){3}}
\put(45,20){\line(1,1){10}}
\put(90,30){\line(1,1){15}}
\put(90,20){\line(1,1){18}}
\put(112,42){\line(1,1){3}}
\put(95,15){\line(1,1){5}}
\put(110,20){\line(1,1){8}}
\put(122,32){\line(1,1){3}}
\put(105,5){\line(1,1){20}}
\put(90,20){\vector(1,-1){15}}
\put(125,25){\vector(-1,1){20}}
\put(125,35){\vector(-1,1){10}}
\put(110,20){\line(1,-1){5}}
\put(90,30){\line(1,-1){3}}
\put(100,20){\line(-1,1){3}}
\put(20,7){\makebox(0,0){$\lambda''$}}
\put(15,43){\makebox(0,0){$\lambda$}}
\put(25,43){\makebox(0,0){$\lambda'$}}
\put(45,17){\makebox(0,0){$\mu''$}}
\put(40,38){\makebox(0,0){$\mu$}}
\put(50,38){\makebox(0,0){$\mu'$}}
\put(75,7){\makebox(0,0){$\lambda''$}}
\put(70,43){\makebox(0,0){$\lambda$}}
\put(80,43){\makebox(0,0){$\lambda'$}}
\put(105,2){\makebox(0,0){$\mu''$}}
\put(115,48){\makebox(0,0){$\mu$}}
\put(105,48){\makebox(0,0){$\mu'$}}
\end{picture}
\end{center}
\caption{Locality (2)}
\label{local2}
\end{figure}
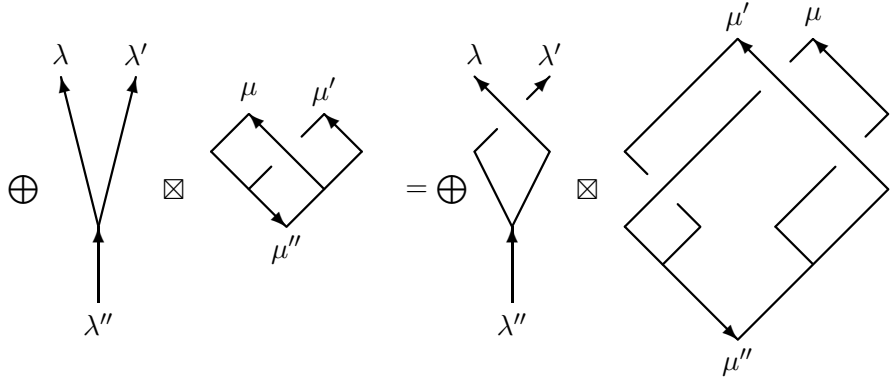

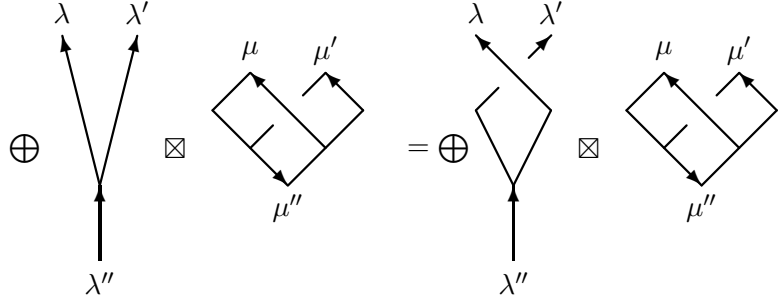
\begin{figure}[htb]
\begin{center}
\unitlength 1mm
\begin{picture}(120,50)
\thicklines
\put(20,10){\vector(0,1){10}}
\put(75,10){\vector(0,1){10}}
\put(20,20){\vector(-1,4){5}}
\put(20,20){\vector(1,4){5}}
\put(75,20){\line(-1,2){5}}
\put(75,20){\line(1,2){5}}
\put(80,30){\vector(-1,1){10}}
\put(70,30){\line(1,1){3}}
\put(77,37){\vector(1,1){3}}
\put(10,25){\makebox(0,0){$\bigoplus$}}
\put(65,25){\makebox(0,0){$=\bigoplus$}}
\put(30,25){\makebox(0,0){$\boxtimes$}}
\put(85,25){\makebox(0,0){$\boxtimes$}}
\put(35,30){\vector(1,-1){10}}
\put(50,25){\vector(-1,1){10}}
\put(55,30){\vector(-1,1){5}}
\put(35,30){\line(1,1){5}}
\put(40,25){\line(1,1){3}}
\put(47,32){\line(1,1){3}}
\put(45,20){\line(1,1){10}}
\put(90,30){\vector(1,-1){10}}
\put(105,25){\vector(-1,1){10}}
\put(110,30){\vector(-1,1){5}}
\put(90,30){\line(1,1){5}}
\put(95,25){\line(1,1){3}}
\put(102,32){\line(1,1){3}}
\put(100,20){\line(1,1){10}}
\put(20,7){\makebox(0,0){$\lambda''$}}
\put(15,43){\makebox(0,0){$\lambda$}}
\put(25,43){\makebox(0,0){$\lambda'$}}
\put(45,17){\makebox(0,0){$\mu''$}}
\put(40,38){\makebox(0,0){$\mu$}}
\put(50,38){\makebox(0,0){$\mu'$}}
\put(75,7){\makebox(0,0){$\lambda''$}}
\put(70,43){\makebox(0,0){$\lambda$}}
\put(80,43){\makebox(0,0){$\lambda'$}}
\put(100,17){\makebox(0,0){$\mu''$}}
\put(95,38){\makebox(0,0){$\mu$}}
\put(105,38){\makebox(0,0){$\mu'$}}
\end{picture}
\end{center}
\caption{Locality (3)}
\label{local3}
\end{figure}

Fig. \ref{local3} represents the locality of
$\Theta_1\otimes_{\C_2}\Theta_2$.
\end{proof}

We consider the irreducible decomposition 
$\Theta_1\otimes_{\C_2}\Theta_2=\bigoplus_i \Theta_3^i$,
which is a finite sum.
By \cite[Corollary 3.6]{BE1}, this coincides with the corresponding
factorial decomposition  $M=\bigoplus_i M_i$ where
the $Q$-system $\Theta_1\otimes_{\C_2}\Theta_2$ corresponds to
an inclusion $N\subset M$ and the one $\Theta_3^i$ corresponds
to $N\subset M_i$.

We first list the following lemma.   See \cite[Definition 5.1]{DMNO}
for the definition of Witt equivalence.

\begin{lemma}
\label{ext}
Let $\C_1,\C_2$ be Witt equivalent modular tensor categories
and $\Theta=(\theta, w, x)$ be an irreducible local
$Q$-system where $\theta$ is of the form
$\bigoplus_{\la\in{\Irr}(\C_1),
\mu\in{\Irr}(\C_2^\opp)}
Z_{\la\mu}\la\boxtimes\bar\mu$.
Then there exists an irreducible local Lagrangian 
$Q$-system $\tilde\Theta=(\tilde\theta, \tilde w, \tilde x)$
where $\tilde \theta$ is of the form
$\bigoplus_{\la\in{\Irr}(\C_1),
\mu\in{\Irr}(\C_2^\opp)}
\tilde Z_{\la\mu}\la\boxtimes\bar\mu$ with
$\tilde Z_{\la\mu}\ge Z_{\la\mu}$ for all 
$\la\in{\Irr}(\C_1),
\mu\in{\Irr}(\C_2^\opp)$ and $\tilde Z_{00}=Z_{00}=1$
where $0$ denotes the identity objects of 
$\C_1$ and $\C_2$.
\end{lemma}

\begin{proof}
Let $\tilde \C$ be the modular tensor category
arising as the ambichiral category from
the $Q$-system $\Theta$ as in \cite[Theorem 4.2]{BEK2}.
(Note that the ambichiral objects correspond to
dyslectic/local modules in the terminology of \cite{DMNO},
\cite{DNO}.)
By \cite[Corollary 4.8]{BEK3}, $\C_1\boxtimes\C_2^\opp$
is Witt equivalent to $\tilde\C$, which means that
$\tilde\C$ is Witt equivalent to the trivial modular
tensor category $\Vec$.  By \cite[Theorem 2.4]{K2},
we have an irreducible local Lagrangian $Q$-system on $\tilde\C$.
Composing this with the original $Q$-system $\Theta$,
we have an irreducible local Lagrangian $Q$-system
$\tilde\Theta$ on $\C_1\boxtimes\C_2^\opp$.
It has the modular invariance property by Proposition
\ref{modular}, and $\tilde Z_{\la\mu}\ge Z_{\la\mu}$
and $\tilde Z_{00}=Z_{00}=1$ are clear.
\end{proof}

\begin{theorem}
If the $Q$-systems $\Theta_1$ and $\Theta_2$ are both Lagrangian,
so is each $\Theta_3^i$.
\end{theorem}

\begin{proof}
Set $Z^3_{\la\nu}=\sum_\mu Z^1_{\la\mu}Z^2_{\mu\nu}$ and
let $\bigoplus_{\la\in{\Irr}(\C_1),\nu\in{\Irr}(\C_3^\opp)}
Z^{3,i}_{\la\nu}\la\boxtimes\bar\nu$ 
be the object for $\Theta_3^i$.  By Proposition \ref{modular},
being Lagrangian for $\Theta_3^i$ is equivalent to
modular invariance property $S_{\C_1} Z^{3,i}=Z^{3,i} S_{\C_3}$
and $T_{\C_1} Z^{3,i}=Z^{3,i} T_{\C_3}$ for $Z^{3,i}$, where
$S_{\C_1}, S_{\C_3}, T_{\C_1}, T_{\C_3}$ are the 
$S$-matrix for $\C_1$, $S$-matrix for $\C_3$, 
$T$-matrix for $\C_1$ and
$T$-matrix for $\C_3$, respectively.

Note that $\C_1$ and $\C_2$ are Witt equivalent, and so are
$\C_2$ and $\C_3$.  Hence $\C_1$ and $\C_3$ are also Witt
equivalent and each $\Theta_3^i$ has a Lagrangian extension
$\tilde\Theta_3^i$ whose object is 
$\bigoplus_{\la\in{\Irr}(\C_1),\nu\in{\Irr}(\C_3^\opp)}
\tilde Z^{3,i}_{\la\nu}\la\boxtimes\bar\nu$ by Lemma \ref{ext}
and we have $S_{\C_1} \tilde Z^{3,i}=\tilde Z^{3,i} S_{\C_3}$
and $T_{\C_1} \tilde Z^{3,i}=\tilde Z^{3,i} T_{\C_3}$ by 
Proposition \ref{modular}.  By Lemma \ref{ext},
we may write $\tilde Z^{3,i}_{\la\nu}=Z^{3,i}_{\la\nu}
+\hat Z^{3,i}_{\la\nu}$, where each $\hat Z^{3,i}_{\la\nu}$
is a non-negative integer.

Since the matrix $\sum_i Z^{3,i}$ also has the modular invariance
property, the matrix $\hat Z^3=\sum_i \hat Z^{3,i}$ also has the modular
invariance property.  This implies
$\sum_{\la\nu} S_{\C_1,0\la} \hat Z^3_{\la\nu}
S_{\C_3,\nu0}=\hat Z^3_{00}$, but $Z^3_{00}=\sum_i Z^{3,i}_{00}=0$ and
$S_{\C_1,0\la}>0$, $S_{\C_3,\nu0}>0$.  We thus have $\hat Z^3_{\la\nu}=0$
for all $\la\in{\Irr}(\C_1)$ and $\nu\in{\Irr}(\C_3)$.  This proves the
modular invariance property $S_{\C_1} Z^{3,i}=Z^{3,i} S_{\C_3}$
and $T_{\C_1} Z^{3,i}=Z^{3,i} T_{\C_3}$ for $Z^{3,i}$, as desired.
\end{proof}

Note that the use of modular invariance in the last paragraph of the
above proof is the same as in \cite[p. 726 (5.2)]{G}.

This relative tensor product of $Q$-systems 
looks similar to that of bimodules, but the example of the
$A_{17}$-$D_{10}$-$E_7$ modular invariants mentioned in
the Introduction shows that their fusion rules do not
give a fusion category since the rigidity axiom is not
satisfied.

We have interpreted an irreducible local Lagrangian $Q$-system
on $\C_1\boxtimes\C_2$ as a gapped domain wall between topological
phases represented with $\C_1$ and $\C_2$ 
in  \cite[Definition 3.1]{K2}. 
(See \cite{HW1}, \cite{HW2}, \cite{LWW} for physical treatments
of gapped domain walls.)
From this viewpoint, the above
relative tensor product gives a mathematical definition
of the composition of gapped domain walls
mentioned in \cite[Fig. 1 (d)]{LWW}.
(Note that irreducibility of a $Q$-system is called
stability of a gapped domain wall in \cite{LWW}.)
A mathematical definition of such a composition has been 
studied in \cite{KZ}, \cite{AKZ}.  It would be interesting
to compare the above definition with theirs.

Another construction of fusion product with some
formal similarity has been defined in \cite{BDH}.
It would be interesting to find direct relations
to their construction.

\end{document}